\documentclass[reqno,centertags,12pt,draft]{amsart}
\usepackage[letterpaper,margin=1.3in]{geometry}
\usepackage{amsmath,amsthm,amsfonts,amssymb,enumerate}

\newcommand{\bb}{\mathbb}
\newcommand{\cc}{\mathcal}
\newcommand{\ol}{\overline}

\newcommand{\bs}{\backslash}

\newcommand{\te}{\theta}

\newcommand{\pd}{\partial}
\newcommand{\supp}{\mathrm{supp}}

\newcommand{\ca}{\mathrm{Cap}}
\newcommand{\x}{{\mathbf x}}

\renewcommand{\Re}{\mathrm{Re}}

\newtheorem{theorem}{Theorem}

\newtheorem{corollary}[theorem]{Corollary}
\theoremstyle{definition}

\theoremstyle{remark}
\newtheorem*{remark}{Remark}
\numberwithin{equation}{section}
\numberwithin{theorem}{section}

\begin{document}
\title{ Widom factors for generalized Jacobi measures}

\author{G\"{o}kalp Alpan}
\address{Department of Mathematics, Uppsala University, Uppsala, Sweden}
\email{gokalp.alpan@math.uu.se}


\begin{abstract}
We study optimal lower and upper bounds for Widom factors \\$W_{\infty,n}(K,w)$ associated with Chebyshev polynomials for the weights \\$w(x)=\sqrt{1+x}$ and $w(x)=\sqrt{1-x}$ on compact subsets of $[-1,1]$. We show which sets saturate these bounds. We consider Widom factors $W_{2,n}(\mu)$ for $L_2(\mu)$ extremal polynomials for measures of the form $d\mu(x)=(1-x)^\alpha (1+x)^\beta  d\mu_K(x)$ where $\alpha+\beta\geq 1$, $\alpha,\beta\in\bb N\cup \{0\}$ and $\mu_K$ is the equilibrium measure of a compact regular set $K$ in $[-1,1]$ with $\pm 1\in K$. We show that for such measures the improved lower bound (which was first studied in \cite{AlpZin20}) $[W_{2,n}(\mu)]^2\geq 2S(\mu)$ holds. For the special cases $d\mu(x)=(1-x^2)d\mu_K(x)$,
$d\mu(x)=(1-x)d\mu_K(x)$, $d\mu(x)=(1+x)d\mu_K(x)$ we determine which sets saturate this lower bound and discuss how saturated lower bounds for $[W_{2,n}(\mu)]^2$ and $W_{\infty,n}(K,w)$ are related.
\end{abstract}

\date{\today}
\maketitle
\nocite{*}

\section{Introduction}
For the basic notions of potential theory, see \cite{Ran95,ST97}. Let $K$ be a compact subset of $\bb R$. We  denote the logarithmic capacity of $K$ by $\ca(K)$ . Let $\ca(K)>0$ and let us denote the equilibrium measure of $K$ by $\mu_K$.  The Green function for the domain $\overline{\bb C}\setminus K$ is given by
\begin{align}
	g_K(z) = -\log\ca(K)+\int\log|z-\zeta|\,d\mu_K(\zeta), \quad z\in\bb C.
\end{align}
Let $\mu$ be a finite (positive) Borel measure with $\supp(\mu)=K$ and let us consider the Lebesgue decomposition of $\mu$ with respect to $\mu_K$: $$d\mu=f d\mu_K+d\mu_s.$$
We define exponential relative entropy function for $\mu$ by
\begin{align}
	S(\mu):=\exp\left[{\int \log{f(x)}\,d\,\mu_K(x)}\right].
\end{align}

We use $\| \cdot \|_K$ to denote the sup-norm on $K$. Let $w$ be a weight function (non-negative and continuous on $K$ and positive on a non-polar subset of $K$) on $K$. Then $n$-th Chebyshev polynomial with respect to $w$ is the minimizer of $\|wP_n\|_K$ over all monic polynomials $P_n$ of degree $n$ and we denote it by $T_{n,w}^{(K)}$. Let
\begin{align}
t_n(K,w):= \|wT_{n,w}^{(K)}\|_K.
\end{align}
We define $n$-th Widom factor for the sup norm with respect to weight $w$ on $K$ by
\begin{align}
W_{\infty,n}(K,w):= \frac{t_n(K,w)}{\ca(K)^n}.
\end{align}

For the weight $w\equiv 1$, we drop $w$ and use the notation $T_n^{(K)}$, $t_n(K)$ and $W_{\infty,n}(K)$ respectively.

 For  a finite (positive) Borel measure $\mu$ with $\supp(\mu)=K$, $n$-th monic orthogonal polynomial for $\mu$ is the minimizer of $\|P_n\|_{L_2(\mu)}$ over all monic polynomials of degree $n$ and we  denote it by $P_n(\cdot;\mu)$. We define $n$-th Widom factor for $\mu$ by
\begin{align}
	W_{2,n}(\mu):= \frac{\|P_n(\cdot;\mu)\|_{L_2(\mu)}}{\ca(K)^n}.
\end{align}
For a deeper discussion of Widom factors we refer the reader to \cite{Alp17,Alp19,AlpGon15,AlpZin20,AlpZin2,And17,AndNaz18,Chr12,CSYZ19,CSZ11,CSZ17,CSZ3,CSZ4,GonHat15,Sch08,SchZin,Tot09,Tot11,Tot14,TotYud15,Wid69}.

For a regular (with respect to the Dirichlet problem) compact subset $K$ of $\bb R$, let $a=\inf K$, $b=\sup K$. Then $K$ can be written as
\begin{align}
	\displaystyle	K=[a,b]\setminus \cup_j I_j,
\end{align}
where the $I_j$'s are open disjoint intervals in $(a,b)$ that are bounded components of  the complement of $K$ in $\bb R$. Let $c_j$ be the critical point of $g_K$ in $I_j$. Then $K$ is called a Parreau-Widom set if
\begin{align}
	\mathrm{PW}(K):=\sum_j g_K(c_j)<\infty.
\end{align}

For the weight $w(x)\equiv 1$, the lower and upper bounds for $W_{\infty,n}(K,w)$ were studied in \cite{CSZ17,Sch08,Tot11} and the final version of these results can be found in \cite[eq.s (1.4), (4.1),  Theorem 1.1, Theorem 4.3]{CSZ3}. Optimal lower and upper bounds of $W_{\infty,n}(K,w)$ for $w(x)=\sqrt{1-x^2}$ have been studied in \cite[Theorem 3]{SchZin}. Our first result (see also Theorem \ref{thm2} for $w(x)=\sqrt{1-x}$) is as follows:

\begin{theorem}\label{thm1}
	Let $K$ be a compact non-polar subset of $[-1,1]$ and 
	\begin{align}
		w(x)=\sqrt{1+x}.
	\end{align}
	Then
	
	\begin{enumerate}[$(i)$]
		\item For each $n\in\bb N$
		\begin{align}\label{cap1}
			W_{\infty,n}(K,w)\geq 2\sqrt{\ca(K)}.
		\end{align}
		Equality is attained in \eqref{cap1} if and only if there exists a polynomial $S_n$ of degree $n$ such that
		\begin{align}\label{leb}
			K=\{z\in\bb C:(1+z)S_n^2(z)\in[0,1]\}.
		\end{align}
		\item In addition, let us assume that $K$ is a regular Parreau-Widom set. Then 
		\begin{align}\label{cap2}
			W_{\infty,n}(K,w)\leq 2 \sqrt{\ca(K)} e^{(1/2)g_K(-1)+ \mathrm{PW}(K)}.
		\end{align}
		Equality is attained  in \eqref{cap2} if and only if $K=[-1,b]$ for some $b\in (-1,1]$.
	\end{enumerate}
	
\end{theorem}

By \cite[Theorem 1.2]{Alp19} (see also \cite[Theorem 2.1]{AlpZin20}) for any finite Borel measure $\mu$ and $n\in \bb N$ the universal lower bound
\begin{align}\label{uni}
	[W_{2,n}(\mu)]^2\geq S(\mu)
\end{align}
holds. Besides, for any compact non-polar set $K$ in $\bb R$ and $n\in \bb N$ 
\begin{align}\label{uni2}
\inf_\mu	[W_{2,n}(\mu)]^2/S(\mu)=1
\end{align}
where the infimum is taken over measures $d\mu(x)=w(x)\,d\mu_K(x)$ with a polynomial $w(x)$ positive on $K$, see  \cite[Theorem 2.2]{AlpZin2}.
On the other hand, we can single out some measures (known examples from Sections 3-5 in \cite{AlpZin20} are $\mu_K$ when $K\subset \bb R$, isospectral torus of a finite gap set, Jacobi measures for certain parameters) such that the  improved lower bound 
\begin{align}\label{uni3}
	[W_{2,n}(\mu)]^2\geq 2 S(\mu)
\end{align}
holds for each $n\in\bb N$. In our second main result, we show that for another family of measures, \eqref{uni3} is satisfied.
\begin{theorem}\label{thm4}
	Let $K$ be a regular compact set such that $K\subset [-1,1]$ and $\pm 1\in K$. Let $\alpha, \beta \in \bb N\cup \{0\}$, $\alpha+\beta\geq 1$ and 
	\begin{align}\label{bor}
		d\mu(x)=(1-x)^{\alpha} (1+x)^\beta d\mu_K(x).
	\end{align}
	Then 
	\begin{align}\label{zor2}
		[W_{2,n}(\mu)]^2\geq 2 S(\mu).
	\end{align}
\end{theorem}
Note that $(1-x)^\alpha (1+x)^\beta \,d\mu_{[-1,1]}=\frac{1}{\pi}(1-x)^{\alpha-1/2}(1+x)^{\beta-1/2}dx\restriction_{[-1,1]}$.  So these measures are indeed Jacobi measures when $K=[-1,1]$ and the measures of the form \eqref{bor} can be considered as generalized Jacobi measures. The four cases $\alpha=0,1$, $\beta=0,1$ are of special importance and we have the following results (see \cite[Sections 3.3 and 4.2]{Mas}, \cite[Section 4]{AlpZin20}):
\begin{theorem}\label{smi}
	Let $T_n, U_n, V_n, W_n$ be Chebyshev polynomials of the first, second, third and fourth kinds respectively, that is $T_n(x)=\cos(n\theta)$, \\$U_n(x)=\sin((n+1)\theta)/\sin(\theta)$ , $V_n(x)=\cos((n+1/2)\theta)/\cos(\theta/2)$, \\$W_n(x)=\sin((n+1/2)\theta)/\sin(\theta/2)$ for $x=\cos{\theta}$,  $0\leq \theta \leq \pi$. Let $K=[-1,1]$. Then
	\begin{align}\label{TNu}
T_{n,w}^{(K)}(x)&=\frac{T_n(x)}{2^{n-1}} \mbox{  for  } w(x)=1,\\
T_{n,w}^{(K)}(x)&=\frac{U_n(x)}{2^{n}} \mbox{  for  } w(x)=\sqrt{1-x^2},\\
T_{n,w}^{(K)}(x)&=\frac{V_n(x)}{2^{n}} \mbox{  for  } w(x)=\sqrt{1+x},\\
T_{n,w}^{(K)}(x)&=\frac{W_n(x)}{2^{n}} \mbox{  for  } w(x)=\sqrt{1-x}
	\end{align}
and $P_n(x; w^2\,d\mu_K)=T_{n,w}^{(K)}(x)$ for the cases $w(x)=1$, $w(x)=\sqrt{1-x^2}$, $w(x)= \sqrt{1+x}$, $w(x)= \sqrt{1-x}$. For each of these four cases 
\begin{align}
	[W_{2,n}(w^2\,d\mu_K)]^2=2S(w^2\,d\mu_K)
\end{align}
is satisfied for all $n\in \bb N$. For all $n\in \bb N$ we also have
	\begin{align}\label{TNunn}
	W_{\infty,n}{(K,w)}&=2  \mbox{  for  } w(x)=1,\\
		W_{\infty,n}{(K,w)}&=1  \mbox{  for  } w(x)=\sqrt{1-x^2},\\
			W_{\infty,n}{(K,w)}&=\sqrt{2}  \mbox{  for  } w(x)=\sqrt{1+x},\\
				W_{\infty,n}{(K,w)}&=\sqrt{2}  \mbox{  for  } w(x)=\sqrt{1-x}.
\end{align}
\end{theorem}
In the next theorem, we determine the sets for which equality is attained in \eqref{uni3} for $d\mu=w^2\,d\mu_K$ for the functions $w(x)=\sqrt{1-x^2}$, $w(x)=\sqrt{1+x}$, $w(x)=\sqrt{1-x}$.
Indeed, we see that for each of these three cases, $W_{\infty,n}{(K,w)}$ and $[W_{2,n}(\mu)]^2$ realize their theoretical lower bounds simultaneously.

\begin{theorem}\label{thm5}
	Let $K$ be a regular compact set such that $K\subset [-1,1]$, $\pm 1 \in K$.
	\begin{enumerate}[$(i)$]
		\item Let 
		\begin{align}
			d\mu(x)=(1-x^2)d\mu_K(x)
		\end{align}
		and
		\begin{align}
			w(x)=\sqrt{1-x^2}.
		\end{align}
		Then $[W_{2,n}(\mu)]^2=2S(\mu)$ if and only if 
		\begin{align}\label{c1}
			K=\{z\in \bb C: (1-z^2)S_n^2(z)\in [0,1]\}
		\end{align}
		for some polynomial $S_n$ of degree $n$ with leading coefficient $c$.
		
		\item Let 
		\begin{align}
			d\mu(x)=(1+x)d\mu_K(x)
		\end{align}
		and
		\begin{align}
			w(x)=\sqrt{1+x}.
		\end{align}	
		Then $[W_{2,n}(\mu)]^2=2S(\mu)$ if and only if 
		\begin{align}\label{c2}
			K=\{z\in \bb C: (1+z)S_n^2(z)\in [0,1]\}
		\end{align}
		for some polynomial $S_n$ of degree $n$ with leading coefficient $c$. 
		
		\item Let 
		\begin{align}
			d\mu(x)=(1-x)d\mu_K(x)
		\end{align}
		and
		\begin{align}
			w(x)=\sqrt{1-x}.
		\end{align}
		Then $[W_{2,n}(\mu)]^2=2S(\mu)$ if and only if 
		\begin{align}\label{c3}
			K=\{z\in \bb C: (1-z)S_n^2(z)\in [0,1]\}
		\end{align}
		for some polynomial $S_n$ of degree $n$ with leading coefficient $c$. 
	\end{enumerate}
	In all three cases, $[W_{2,n}(\mu)]^2=2S(\mu)$ implies
	\begin{align}\label{e1}
		P_n(x;\mu)=\frac{S_n(x)}{c}=T_{n,w}^{(K)}(x).
	\end{align}
\end{theorem}

\begin{remark}Analogous results for $w(x)\equiv 1$ can be found in \cite[Theorem 4.4 and Corollary 4.5]{AlpZin2}). Theorem \ref{thm4}, Theorem \ref{thm5} for $K=[-1,1]$, \cite[Theorem 3]{SchZin} and \cite[Theorem 4.4 and Corollary 4.5]{AlpZin2}) together  imply Theorem \ref{smi}.
\end{remark} 
The plan for the rest of the paper is as follows. We prove Theorem \ref{thm1}, Theorem \ref{thm2} in Section 2 and prove Theorem \ref{thm4} and Theorem \ref{thm5} in Section 3. We also show that a more general version of Theorem \ref{thm4} and Theorem \ref{thm5} hold in Section 3.

\section{Chebyshev polynomials on subsets of $[-1,1]$ for weight functions $w(x)=\sqrt{1+x}$ and $w(x)=\sqrt{1-x}$}
Let $Q$ be a polynomial with leading coefficient $c$ and $\deg Q=N\geq 2$ such that $K:= Q^{-1}([a,b])=\{z\in\bb C:\,Q(z)\in[a,b]\}\subset \bb R$ and $a<b$. Let $\cc T(x)=\frac{2 Q(x)-a-b}{b-a}$. Then $\cc T^{-1}([-1,1])=K$, $\deg \cc T=N$. In view of \cite[Lemma 1, Theorem 11]{GerVAs88} and \cite[Corollary 2.3.]{Peh03}, there are $N$ closed intervals $E_1, \ldots,  E_N$ such that $E_k\cap E_j$ contains at most one point for any $j\neq k$ and $Q(E_j)=[a,b]$, $\cc T(E_j)=[-1,1]$. Here, 
\begin{align}\label{invpre}
	\mu_K (E_j)= 1/N
\end{align}
and $\cc T$ has real simple zeros and real coefficients and thus $Q$ is also a real polynomial and in particular $c\in \bb R$.

For the rest of this section we assume that $K$ is a non-polar compact subset of $[-1,1]$. We consider Chebyshev polynomials on $K$ with respect to the weights  $w(x)=\sqrt{1+x}$ and $w(x)=\sqrt{1-x}$. We focus on  $w(x)=\sqrt{1+x}$ and the results on $w(x)=\sqrt{1-x}$ follow immediately by symmetry about $0$.
Let 
\begin{align}
	w(x)=\sqrt{1+x}.
\end{align}
For each $n\in \bb N$, there are $n+1$ points (not necessarily unique) 
\begin{align}\label{p11}
x_0<x_1<\dots<x_n,
\end{align}
$x_k\in K$, $k=0,\ldots,n$ such that
\begin{align}\label{p12}
w(x_k)	T_{n,w}^{(K)}(x_k)=(-1)^{n-k}\|	w T_{n,w}^{(K)}\|_K
\end{align}
is satisfied, see e.g. Section 4.2 in \cite{Lo}. It is easy to see that the zeros of $T_{n,w}^{(K)}$ are real, let us denote them by $y_1<\dots <y_n$. Let $y_0:=-1$. Then by the intermediate value theorem, \eqref{p11}, \eqref{p12} we get
\begin{align}\label{xnyn}
-1=y_0<x_0<y_1<x_1<\dots <y_n<x_n\leq 1.
\end{align}
Let us consider the polynomial
\begin{align}\label{Q}
Q_{2n+1}(x):= (w(x)T_{n,w}^{(K)}(x))^2=(1+x)(T_{n,w}^{(K)}(x))^2.
\end{align}
Then it follows from \eqref{xnyn} that each point in the interval $(0,t_n(K,w)^2)$ has $2n+1$ distinct preimages of $Q_{2n+1}$ in $(y_k,x_k)$, $k=0,\ldots, n$ and in $(x_k,y_{k+1})$, $k=0,\ldots,n-1$. Thus 
\begin{align}\label{kn}
	K_n:= Q_{2n+1}^{-1}([0,t_n(K,w)^2])\subset [-1,1],
\end{align}
\begin{align}
	K\subset K_n
\end{align}
and $K_n$ is a union of finitely many closed intervals. In view of \eqref{invpre}, we have
\begin{align}
\mu_{K_n}([y_k,x_k])=\mu_{K_n}([x_j,y_{j+1}])=\frac{1}{2n+1},\,\, k=0,\ldots,n\mbox{ and } j=0,\ldots,n-1.
\end{align}
Since the leading coefficient of $Q_{2n+1}$ is $1$, by  \cite[Theorem 5.2.5]{Ran95}, 
\begin{align}\label{cakn}
\ca(K_n)^{2n+1}=\frac{t_n(K,w)^2}{4}.
\end{align}

\begin{proof}[Proof of Theorem \ref{thm1}]
	Let $Q_{2n+1}$ be as in \eqref{Q}, $K_n$ be as in \eqref{kn} and $x_k$, $y_k$ be as described above and so satisfy \eqref{p11}, \eqref{p12}, \eqref{xnyn}.
		\begin{enumerate}[$(i)$]
		\item  Since $K\subset K_n$, it follows from \eqref{cakn} that
		\begin{align}\label{caeq}
\frac{t_n(K,w)^2}{4}=\ca(K_n)^{2n+1}\geq\ca(K)^{2n+1}
		\end{align}
		and thus \eqref{cap1} holds.
		
	Let us assume that equality is satisfied in \eqref{cap1}. Then by \eqref{caeq}, $\ca(K)=\ca(K_n)$ and by \cite[Lemma 13]{SchZin}, $K=K_n$. If we let $S_n= T_{n,w}^{(K)}/t_n(K,w)$ then \eqref{leb} is satisfied.
	
	Conversely let us assume that \eqref{leb} is satisfied and let $c$ be the leading coefficient of $S_n$ and define $R_n=S_n/c$. Then 
	\begin{align}\label{sne1}
		\ca(K)^{2n+1}=\frac{1} {4|c|^2}
	\end{align}
	and 
	\begin{align}\label{sne2}
		\|wR_n\|_K=\frac{1}{|c|}.
	\end{align}
Combining \eqref{sne1}, \eqref{sne2} we get
			\begin{align}
W_{\infty, n}(K,w)\leq 2 \sqrt{\ca(K)}.
		\end{align}
Since \eqref{cap1} is satisfied this implies equality in \eqref{cap1}.

\item Note that (see \cite[Lemma 13]{SchZin}) 
\begin{align}\label{right}
	\log\left (\frac{\ca(K_n)}{\ca(K)}\right)=\int_{K_n\setminus K} g_K(x )\,d\mu_{K_n}(x)\leq \int_{[-1,1]\setminus K}g_K(x)\,d\mu_{K_n}(x).
\end{align}
Let $a=\inf K$ and $b=\sup K$ and 

\begin{align}\label{left}
	\displaystyle	[-1,1]\setminus K=[-1,a) \cup (b,1]\cup_j I_j
\end{align}
where the $I_j$'s are open disjoint  intervals that are connected components of $[-1,1]\setminus K$. Since $[-1,a)\cap K_n\subset [y_0,x_0]$ and $\sup_{x\in [-1,a)}g_K(x)=g_K(-1)$, 
it follows from \eqref{invpre} that
\begin{align}\label{left1}
\int_{[-1,a)} g_K(x)\,d\mu_{K_n}(x)\leq \frac{1}{2n+1}g_K(-1).
\end{align}
Since $Q_{2n+1}$ is strictly increasing on $[y_n,1]$, $b=x_n$. Thus $(b,1]\cap K_n=\emptyset$ and
\begin{align}\label{left2}
	\int_{(b,1]} g_K(x)\,d\mu_{K_n}(x)=0.
\end{align}
Since $x_k\in K$ for each $k$, the interval $I_j$ (for any particular $j$) can intersect with at most two consecutive intervals $[y_k,x_{k}]$, $[x_k,y_{k+1}]$. Thus
\begin{align}\label{left3}
	\int_{I_j} g_K(x)\,d\mu_{K_n}(x)\leq \frac{2}{2n+1}g_K(c_j)
\end{align}
where $c_j$ is the critical point of $g_K$ in $I_j$. Thus \eqref{right}, \eqref{left1}, \eqref{left2}, \eqref{left3} yield

\begin{align}\label{right1}
	\log\left (\frac{\ca(K_n)}{\ca(K)}\right)\leq \frac{2}{2n+1}\mathrm{PW}(K)+\frac{1}{2n+1}g_K(-1)
\end{align}
and thus
\begin{align}\label{right2}
	\left (\frac{\ca(K_n)}{\ca(K)}\right)^{\frac{2n+1}{2}}\leq e^{\mathrm{PW}(K)+(1/2)g_K(-1)}.
\end{align}
Since \eqref{cakn} holds, \eqref{cap2} follows from \eqref{right2}. 

Since $g_K$ is  piecewise strictly monotone on all the intervals given on the right side of \eqref{left}, all the inequalities in \eqref{left3} and \eqref{left1}  are strict. Therefore equality is attained in \eqref{cap2} only if  all the $I_j$'s and $[-1,a)$ are empty. Thus equality is satisfied in \eqref{cap2} only if $K=[-1,b]$ for some $b\in(-1,1]$.

Conversely, if $K=[-1,b]$ for some $b\in(-1,1]$ then for all $n\in \bb N$, $K=K_n$  and \eqref{leb} holds. This implies 
\begin{align}\label{right3}
W_{\infty,n}(K,w)=2\sqrt{\ca(K)}= 2\sqrt{\ca(K)}e^{\mathrm{PW}(K)+(1/2)g_K(-1)}
\end{align}
since $g_K(-1)=\mathrm{PW}(K)=0$. This completes the proof.
\end{enumerate}
\end{proof}
Now, let us consider the above theorem for the weight $w(x)=\sqrt{1-x}$ on $K\subset [-1,1]$. Let $L=\{z\in\bb C:  -z\in K\}$ and $w_1(x)=\sqrt{1+x}$. If $T_{n,w}^{(K)}(x)=\prod_{j=1}^n (x-\alpha_j)$ then  $T_{n,w_1}^{(L)}(x)=\prod_{j=1}^n (x+\alpha_j)$, $\ca(K)=\ca(L)$, $t_{n}(K,w)=t_{n}(L,w_1)$ and $g_K(z)=g_L(-z)$. Thus we can rewrite Theorem \ref{thm1} for $w(x)=\sqrt{1-x}$ as follows:
\begin{theorem}\label{thm2}
	Let $K$ be a compact non-polar subset of $[-1,1]$ and 
	\begin{align}
		w(x)=\sqrt{1-x}.
	\end{align}
	Then
	
	\begin{enumerate}[$(i)$]
		\item For each $n\in\bb N$
		\begin{align}\label{cap11}
			W_{\infty,n}(K,w)\geq 2\sqrt{\ca(K)}.
		\end{align}
		Equality  is attained \eqref{cap11} if and only if there exists a polynomial $S_n$ of degree $n$ such that
		\begin{align}\label{leb22}
			K=\{z\in\bb C:(1-z)S_n^2(z)\in[0,1]\}.
		\end{align}
		\item In addition, let us assume that $K$ is a regular Parreau-Widom set . Then 
		\begin{align}\label{cap22}
			W_{\infty,n}(K,w)\leq 2 \sqrt{\ca(K)} e^{(1/2)g_K(1)+ \mathrm{PW}(K)}.
		\end{align}
		Equality is attained in \eqref{cap22} if and only if $K=[a,1]$ for some $a\in [-1,1)$.
	\end{enumerate}
	\end{theorem}
The next theorem is used in the proof of Theorem \ref{thm5}.
\begin{theorem}\label{thm3}
	\begin{enumerate}[$(i)$]
		\item Let 	$K=\{z\in\bb C:(1+z)S_n^2(z)\in[0,1]\}\subset [-1,1]$ where $S_n$ is a polynomial of degree $n$ with leading coefficient $c$ and $n\geq 1$. Then
		\begin{align}\label{pes}
			T_{2n+1}^{(K)}(z)=\frac{(1+z)S_n^2(z)}{c^2}-2\ca(K)^{2n+1}.
		\end{align}
			\item Let 	$K=\{z\in\bb C:(1-z)S_n^2(z)\in[0,1]\}\subset [-1,1]$ where $S_n$ is a polynomial of degree $n$ with leading coefficient $c$ and $n\geq 1$. Then
	\begin{align}
		T_{2n+1}^{(K)}(z)=2\ca(K)^{2n+1}-\frac{(1-z)S_n^2(z)}{c^2}.
	\end{align}
	\end{enumerate}
\end{theorem}
\begin{proof}
	\begin{enumerate}[$(i)$]
		\item Note that $\ca(K)^{2n+1}=1/(4|c|^2)$. Let $w(x)=\sqrt{1+x}$ and $R_n=c^{-1}S_n$. Then $\|wR_n\|_K=1/|c|=2\ca(K)^n\sqrt{\ca(K)}$. Thus \\$\|wR_n\|_K/\ca(K)^n= 2\sqrt{\ca(K)}$. It follows from \eqref{cap1} and uniqueness of $T_{n,w}^{(K)}$ that $T_{n,w}^{(K)}=R_n$. Therefore the polynomial on the right side of the equality in \eqref{pes} is a monic polynomial of degree $2n+1$ and its sup-norm on $K$ is $2\ca(K)^{2n+1}.$ Hence (see \cite{Sch08}) \eqref{pes} holds.
		
		The part $(ii)$ follows from a similar argument used in $(i)$.
	\end{enumerate}
\end{proof}

\section{Widom factors for generalized Jacobi weights}
 For a function $f$ defined on $\bb D$ we use $f(e^{i\theta})$ to denote non-tangential boundary value of $f$ at  $e^{i\theta}$.
Let us briefly mention some results regarding universal covering maps, see \cite[Chapter 7]{Nev70}, \cite[Chapter 16]{Mar19}, \cite[Chapter 9]{Sim11}.
Let $K$ be a regular compact subset of $\bb R$. Then there is a unique covering map $\x:\bb D\to\ol{\bb C}\bs K$ such that it is a meromorphic function which satisfies
\begin{align}
	\x(0)=\infty,
\end{align}
\begin{align}
	\lim_{z\rightarrow 0} z\x(z)>0.
\end{align}
In addition, $\x(e^{i\theta})\in K$ ($d\theta$ a.e.) and (see e.g. Section~2.4 in \cite{Fish}) 
\begin{align}
	\int f d\mu_K = \int f(\x({e^{i\te}}))\, \frac{d\te}{2\pi},
	\quad f\in L_1(\mu_K).
\end{align}
We consider the analytic function $B$ on $\bb D$ that is  uniquely determined by

\begin{align}\label{grin}
	|B(z)|= e^{-g_K(\x(z))},
\end{align}
\begin{align}\label{xBcap}
	\lim_{z\to0}\x(z)B(z)=\ca(K).
\end{align}
It is given as a Blaschke product (\cite[Theorem 16.11]{Mar19}),  $|B(e^{i\te})|=1$ a.e.\ on $\pd\bb D$, it has simple zeros at $\x^{-1}(\infty)$. For a polynomial $P$ of degree $n$, the function $F(z)=P(\x(z))B^n(z)$  has only removable singularities and  it can be identified with a bounded analytic function on $\bb D$. For such functions, we assume that removable singularities have already
been removed.

Since $B(0)=0$, $\overline{B(e^{i\theta})}={B(e^{i\theta})}^{-1}$, for all $k\in \bb Z\setminus\{0\}$ we have
\begin{align}\label{bbb1}
	\int_{0}^{2\pi}B^k(e^{i\theta})\frac{d\theta}{2\pi}=0.
\end{align}

\begin{proof}[Proof of Theorem \ref{thm4}]
	By Frostman's theorem (see \cite[Theorem 3.3.4, Theorem 4.2.4]{Ran95})
		\begin{align}
S(\mu)&=\exp\left[\int (\alpha \log|1-x| +\beta \log|1+x|)\, d\mu_K(x)   \right]\nonumber\\
&=\exp[(\alpha+\beta)\log(\ca(K))]\nonumber\\
&=(\ca(K))^{\alpha+\beta}.\label{zor1}
	\end{align}

Note that, $(-1)^\alpha (1-\x(z))^{\alpha}(1+\x(z))^{\beta}B^{\alpha+\beta}(z)$ is bounded and analytic on $\bb D$. It is never zero on $\bb D$ since $-1,1\in K$  and \eqref{xBcap} is satisfied.
Therefore it has an analytic square root $H$ on $\bb D$ such that $H(0)=\ca(K)^{\frac{\alpha+\beta}{2}}$. Let $R_n$ be a monic real polynomial of degree $n$. Then 
\begin{align}\label{weq1}
	(\ca(K))^{\frac{2n+\alpha+\beta}{2}}=\int _{0}^{2\pi}H(e^{i\theta})R_n(\x(e^{i\theta}))B^n(e^{i\theta})\,\frac{d\theta}{2\pi}.
\end{align}
The right hand side of \eqref{weq1} can be written as 
\begin{align}\label{weq2}
\int _{0}^{2\pi}\left[(1-\x(e^{i\theta}))^\alpha (1+\x(e^{i\theta}) )^\beta\right]^{\frac{1}{2}} R_n(\x(e^{i\theta}))\left[(-1)^\alpha B^{\alpha+\beta}(e^{i\theta})\right]^{\frac{1}{2}} B^n(e^{i\theta}) \,\frac{d\theta}{2\pi}
\end{align}
where $G(e^{i\theta})=\left[(1-\x(e^{i\theta}))^\alpha (1+\x(e^{i\theta}) )^\beta\right]^{\frac{1}{2}}$ is non-negative on $\partial \bb D$ and \\
$\left[(-1)^\alpha B^{\alpha+\beta}(e^{i\theta})\right]^{\frac{1}{2}}=H(e^{i\theta})/G(e^{i\theta})$.

Since $\x(e^{i\theta})\in \bb R$ and $R_n$ is a real polynomial, if we add the term in \eqref{weq2} to its conjugate we get
\begin{align}\label{weq3}
2(\ca(K))^{\frac{2n+\alpha+\beta}{2}}&=\int _{0}^{2\pi}\left[(1-\x(e^{i\theta}))^\alpha (1+\x(e^{i\theta}) )^\beta\right]^{\frac{1}{2}} R_n(\x(e^{i\theta}))\nonumber\\
&\times \left[\left((-1)^\alpha B^{\alpha+\beta}(e^{i\theta})\right)^{\frac{1}{2}} B^n(e^{i\theta})+\overline{\left((-1)^\alpha B^{\alpha+\beta}(e^{i\theta})\right)^{\frac{1}{2}} B^n(e^{i\theta})} \right]\,\frac{d\theta}{2\pi}.
\end{align}
Since $\left((-1)^\alpha B^{\alpha+\beta}(e^{i\theta})\right)^{\frac{1}{2}} B^n(e^{i\theta})\in \partial \bb D$ and $\overline{B(e^{i\theta})}={(B(e^{i\theta}))}^{-1}$ we have 
\begin{align}\label{weq4}
 &\left[\left((-1)^\alpha B^{\alpha+\beta}(e^{i\theta})\right)^{\frac{1}{2}} B^n(e^{i\theta})+\overline{\left((-1)^\alpha B^{\alpha+\beta}(e^{i\theta})\right)^{\frac{1}{2}} B^n(e^{i\theta})} \right]^2\nonumber\\
 &= 2+2\Re{\left[(-1)^\alpha (B(e^{i\theta}))^{2n+\alpha+\beta}\right]}.
\end{align}
Using Cauchy-Schwarz in \eqref{weq3} and then \eqref{weq4} we obtain
\begin{align}
	&2(\ca(K))^{\frac{2n+\alpha+\beta}{2}}\nonumber \\
	&\leq \left[\int _{0}^{2\pi}\left[(1-\x(e^{i\theta}))^\alpha (1+\x(e^{i\theta}) )^\beta\right] R_n^2(\x(e^{i\theta}))\,\frac{d\theta}{2\pi}\right]^{\frac{1}{2}}\nonumber\\
	&\times \left[\int _{0}^{2\pi}\left(\left((-1)^\alpha B^{\alpha+\beta}(e^{i\theta})\right)^{\frac{1}{2}} B^n(e^{i\theta})+\overline{\left((-1)^\alpha B^{\alpha+\beta}(e^{i\theta})\right)^{\frac{1}{2}} B^n(e^{i\theta})} \right)^2\,\frac{d\theta}{2\pi}\right]^{\frac{1}{2}}\label{aa1}\\
	&= \left[\int _{0}^{2\pi}\left[(1-\x(e^{i\theta}))^\alpha (1+\x(e^{i\theta}) )^\beta\right] R_n^2(\x(e^{i\theta}))\,\frac{d\theta}{2\pi}\right]^{\frac{1}{2}}\nonumber\\
	&\times \left[\int _{0}^{2\pi}\left(2+2\Re{\left[(-1)^\alpha (B(e^{i\theta}))^{2n+\alpha+\beta}\right]}\right)\,\frac{d\theta}{2\pi}\right]^{\frac{1}{2}}\\
	&= \left[\int (1-x)^\alpha (1+x )^\beta R_n^2(x)\, d \mu_K(x)\right]^{\frac{1}{2}} \left[2+2\Re{\left[(-1)^\alpha (B(0))^{2n+\alpha+\beta}\right]}\right]^{\frac{1}{2}}\\
	&=\sqrt{2}\|R_n\|_{L_2({\mu})}.
\end{align}
 Hence, it follows from \eqref{zor1} that 
 \begin{align}
 	\frac{\|R_n\|^2_{L_2({\mu})}}{\ca{(K)}^{2n}}\geq 2 S(\mu)
 	\end{align}
and this implies \eqref{zor2} since $P_n(\cdot;\mu)$ is a real polynomial.
\end{proof}

\begin{proof}[Proof of Theorem \ref{thm5}]
	Since $K$ is regular, $\mathrm{supp}(\mu_K)=K$ in view of \cite[Theorem 4.2.3]{Ran95} and \cite[Corollary 5.5.12]{Sim11}. Thus $\ca{(K)}=\ca{(\mathrm{supp}(\mu))}$ for all three cases.
	\begin{enumerate}[$(i)$]
		\item Let $\alpha=1$, $\beta=1$ in Theorem \ref{thm4}. 
		First, let us assume $[W_{2,n}(\mu)]^2=2S(\mu)$. Then equality is attained in \eqref{aa1}. Therefore there is a monic real polynomial $R_n$ of degree $n$ and a constant $d$ such that
		\begin{align}\label{b1}
			(1-\x^2(e^{i\theta}))R_n^2(\x(e^{i\theta}))=\frac{d(1-B^{2n+2}(e^{i\theta}))^2}{-B^{2n+2}(e^{i\theta})},\,\,\, ( d\theta\, \mbox{a.e.}).
		\end{align}
 By \eqref{zor2} and uniqueness of $n$-th monic orthogonal  polynomial this implies
\begin{align}\label{orti}
	P_n(x;\mu)=R_n(x).
\end{align} 

Let  $F_1(z)=-(1-\x^2(z)) R_n^2(\x(z))B^{2n+2}(z)$ and $F_2(z)=d(1-B^{2n+2}(z))^2$ be two functions on $\bb D$. Then by \eqref{b1}, $F_1(e^{i\theta})=F_2(e^{i\theta})$  on $ \partial \bb D$ and since $F_1$ and $F_2$ are bounded and analytic, this implies that $F_1(z)=F_2(z)$, $z\in\bb D$. Since $F_1(0)=\ca(K)^{2n+2}$ and $B(0)=0$ this implies that $d=\ca(K)^{2n+2}$. Thus for all $z\in \bb D$ we have
\begin{align}
	|F_1(z)|&= |(1-\x^2(z))R_n^2(\x(z))|{e^{-(2n+2)g_K(\x(z))}}\\
	&=|F_2(z)|\\
	&\leq  {\ca(K)^{2n+2}\left[1+e^{-(2n+2)g_K(\x(z))}\right]^2}.
\end{align}
Hence for all $z\in \bb C\setminus K$, we have 
\begin{align}\label{belieber}
	|(1-z^2)R_n^2(z)|\leq  \frac{\ca(K)^{2n+2}\left[1+e^{-(2n+2)g_K(z)}\right]^2}{e^{-(2n+2)g_K(z)}}.
\end{align}
Since $g_K$ is continuous throughout $\bb C$ and $g_K(x)=0$ for $x\in K$ by regularity  of $K$, it follows from \eqref{belieber} that
\begin{align}\label{belieber2}
\displaystyle	\sup_{z\in K}\left[|1-z^2|R_n^2(z)\right]\leq  4\ca(K)^{2n+2}.
\end{align}
The inequality \eqref{belieber2} implies $W_{\infty,n}(K,w)\leq 2 {\ca(K)}$ and combining this with \cite[eq. (12)]{SchZin} we obtain $W_{\infty,n}(K,w)= 2 {\ca(K)}$. By \cite[Theorem 3]{SchZin} this implies \eqref{c1}. Note that \eqref{belieber2} implies $R_n=T_{n,w}^{(K)}$ and by the proof of \cite[Theorem 3]{SchZin}, $T_{n,w}^{(K)}=S_n/c$. Combining this with \eqref{orti} we get \eqref{e1}.

Conversely, let us assume that $K$ satisfies \eqref{c1}. Let $R_n=S_n/c$. Then  
\begin{align}\label{f1}
R_n=T_{n,w}^{(K)}.
\end{align}

Using \cite[Theorem 12(iii)]{SchZin} in \eqref{g1} and \cite[eq.s (2.3), (2.4)]{CSZ17} in \eqref{g2} we see that for all $z\in\bb D$,
\begin{align}
2\ca(K)^{2n+2}-(1-\x^2(z))R_n^2(\x(z))&=T^{(K)}_{2n+2}(\x(z))\label{g1}\\
&=\ca(K)^{2n+2}\left[B^{2n+2}(z)+B^{-2n-2}(z)\right]\label{g2}.
\end{align}
Hence
\begin{align}\label{ba1}
	\ca(K)^{2n+2}\left[2-B^{2n+2}(z)-B^{-2n-2}(z)\right]= (1-\x^2(z))R_n^2(\x(z)).
\end{align}
Note that in view of \eqref{bbb1}
\begin{align}
	\ca(K)^{2n+2}	\int_0^{2\pi}\left[2-B^{2n+2}(e^{i\theta})-B^{-2n-2}(e^{i\theta})\right]\frac{d\theta}{{2\pi}}=2\ca(K)^{2n+2}
\end{align}
and
\begin{align}
\int_0^{2\pi}(1-\x^2(e^{i\theta}))R_n^2(\x(e^{i\theta}))\frac{d\theta}{{2\pi}}&=\int (1-x^2)R_n^2(x)d\mu_K(x)\\
&=\|R_n\|^2_{L_2(\mu)}.
\end{align}
Thus, it follows from \eqref{ba1} and \eqref{zor1} that 
\begin{align}
\frac{\|R_n\|^2_{L_2(\mu)}}{\ca(K)^{2n}}=2S(\mu).
\end{align}
This implies  $[W_{2,n}(\mu)]^2\leq 2S(\mu)$ and it follows from \eqref{zor2} that $[W_{2,n}(\mu)]^2= 2S(\mu)$.
\item Let $\alpha=0$, $\beta=1$ in Theorem \ref{thm4}. First, let us assume  that $[W_{2,n}(\mu)]^2= 2S(\mu).$
Then equality is attained in \eqref{aa1} and thus there is a monic real polynomial $R_n$ of degree $n$ and a constant $d$ such that

	\begin{align}\label{r1}
		(1+\x(e^{i\theta}))R_n^2(\x(e^{i\theta}))=\frac{d \left[B^{2n+1}(e^{i\theta})+1\right]^2}{B^{2n+1}(e^{i\theta})}\,\, (d\theta \mbox{ a.e.})
	\end{align}
and
\begin{align}\label{p1}
	R_n(x)=P_n(x;\mu).
\end{align}

 Let $F_1(z)= (1+\x(z))R_n^2(\x(z))B^{2n+1}(z)$ and $F_2(z)=d (B^{2n+1}(z)+1)^2$. Then $F_1$, $F_2$ are bounded analytic functions on $\bb D$ and $F_1(e^{i\theta})=F_2(e^{i\theta})$
on $\partial \bb D$ and thus $F_1(z)=F_2(z)$ for $z\in \bb D$. Since $F_1(0)=\ca(K)^{2n+1}$ and $B(0)=0$ we get $d=\ca(K)^{2n+1}$. For $z\in \bb D$,
\begin{align}
	|F_1(z)|&= |(1+\x(z))R_n^2(\x(z))|{e^{-(2n+1)g_K(\x(z))}}\\
	&=|F_2(z)|\\
	&\leq  {\ca(K)^{2n+1}\left[1+e^{-(2n+1)g_K(\x(z))}\right]^2}
\end{align}
is satisfied. Hence for $z\in \bb C\setminus K$, 
\begin{align}\label{r2}
	|(1+z)R_n^2(z)|\leq  \frac{\ca(K)^{2n+1}\left[1+e^{-(2n+1)g_K(z)}\right]^2}{e^{-(2n+1)g_K(z)}}.
\end{align}
Since $g_K(x)=0$ for $x\in K$ and $g_K$ is continuous in $\bb C$, \eqref{r2} implies that
\begin{align}\label{belieber3}
	\displaystyle	\sup_{z\in K}\left[|1+z|R_n^2(z)\right]\leq  4\ca(K)^{2n+1}.
\end{align}
This implies that $W_{\infty,n}(K,w)\leq 2 \sqrt{\ca(K)}$ and since \eqref{cap1} holds we have $W_{\infty,n}(K,w)= 2 \sqrt{\ca(K)}$ and $R_n=T_{n,w}^{(K)}$. It follows from Theorem \ref{thm1} (i) that $K$ satisfies \eqref{c2}. It follows from the proof of Theorem \ref{thm3} (i) that $R_n=S_n/c$. This combined with \eqref{p1} implies \eqref{e1}.

Conversely, assume that $K$ satisfies \eqref{c2} and let $R_n=S_n/c$. Then 
\begin{align}\label{o1}
	R_n=T_{n,w}^{(K)}.
\end{align}

Using  Theorem \ref{thm3} (i) in \eqref{g11} and \cite[eq.s (2.3), (2.4)]{CSZ17} in \eqref{g12} we see that for all $z\in\bb D$,
\begin{align}
	(1+\x(z))R_n^2(\x(z))-2\ca(K)^{2n+1}&=T^{(K)}_{2n+1}(\x(z))\label{g11}\\
	&=\ca(K)^{2n+1}\left[B^{2n+1}(z)+B^{-2n-1}(z)\right]\label{g12}.
\end{align}
Therefore
\begin{align}\label{debu}
	(1+\x(z))R_n^2(\x(z))=\ca(K)^{2n+1}\left[B^{2n+1}(z)+B^{-2n-1}(z)+2\right].
\end{align}
Here,
\begin{align}
	\ca(K)^{2n+1}	\int_0^{2\pi}\left[B^{2n+1}(e^{i\theta})+B^{-2n-1}(e^{i\theta})+2\right]\frac{d\theta}{{2\pi}}=2\ca(K)^{2n+1}
\end{align}
and
\begin{align}
	\int_0^{2\pi}(1+\x(e^{i\theta}))R_n^2(\x(e^{i\theta}))\frac{d\theta}{{2\pi}}&=\int (1+x)R_n^2(x)d\mu_K(x)\\
	&=\|R_n\|^2_{L_2(\mu)}.
\end{align}
Hence in view of  \eqref{zor1} and \eqref{debu} we have 
\begin{align}
	\frac{\|R_n\|^2_{L_2(\mu)}}{\ca(K)^{2n}}=2S(\mu)
\end{align}
which implies  $[W_{2,n}(\mu)]^2\leq 2S(\mu)$ and using \eqref{zor2} we see that $[W_{2,n}(\mu)]^2= 2S(\mu)$.

\item Let $L=\{z\in \bb C: -z\in K\}$ and $d\mu_2(x)=(1+x)d\mu_L(x).$ Then in view of  \cite[eq. (1.9) and Lemma 4(c)]{PehSte00} we see that $P_n(x;\mu)=(-1)^n P_n(-x;\mu_2)$ and
$\|P_n(\cdot;\mu)\|_{L_2(\mu)}=\|P_n(\cdot;\mu_2)\|_{L_2(\mu_2)}$. Therefore part $(iii)$ is a straightforward consequence of part $(ii)$.
	\end{enumerate}
\end{proof}
We can easily generalize Theorems \ref{thm4} and \ref{thm5} as follows.

\begin{corollary}\label{cor}
	Let $L$ be a regular compact subset of $[a,b]$, $a,b\in L$ and $\alpha,\beta \in \bb N\cup \{0\}$ with $\alpha+\beta\geq 1$. Let
	\begin{align}
		T(x)=\frac{2x-a-b}{b-a},
	\end{align}
	\begin{align}
d\nu(x)=(1-T(x))^\alpha(1+T(x))^\beta  d\mu_L(x),
\end{align}
and
	\begin{align}
	w(x)=\sqrt{(1-T(x))^\alpha(1+T(x))^\beta }.
\end{align}
Then
\begin{enumerate}[$(i)$]
	\item \begin {align}\label{z1}
		\left[W_{2,n}(\nu)\right]^2\geq 2S(\nu).
	\end{align}
	\item Let $\alpha=1$, $\beta=1$. Then equality is satisfied in \eqref{z1} if and only if there is a polynomial $S_n$ of degree $n$ such that
\begin {align}\label{z2}
L=\{z\in \bb C: (1-T^2(z))S_n^2(T(z))\in [0,1]\}.
\end{align}
\item Let $\alpha=0$, $\beta=1$. Then equality is satisfied in \eqref{z1} if and only if there is a polynomial $S_n$ of degree $n$ such that
\begin {align}\label{z3}
L=\{z\in \bb C: (1+T(z))S_n^2(T(z))\in [0,1]\}.
\end{align}
\item Let $\alpha=1$, $\beta=0$. Then equality is satisfied in \eqref{z1} if and only if there is a polynomial $S_n$ of degree $n$ such that
\begin {align}\label{z4}
L=\{z\in \bb C: (1-T(z))S_n^2(T(z))\in [0,1]\}.
\end{align}
\end{enumerate}
If  \eqref{z2} (or \eqref{z3} or \eqref{z4} respectively) is satisfied then
\begin{align}\label{z5}
	P_n(x;\nu)=\frac{1}{c}\left(\frac{b-a}{2}\right)^nS_n(T(x))=T_{n,w}^{(L)}(x)
\end{align}
where $c$ is the leading coefficient of $S_n$.
\end{corollary} 
\begin{proof}
	Let $K=T(L)$. This implies that $L=T^{-1}(K)$, $K\subset [-1,1]$, $\pm 1\in K$ and $K$ is also regular. Let
	\begin{align}
		d\mu(x)=(1-x)^\alpha (1+x)^\beta d\mu_K(x).
	\end{align}
Then in view of \cite[Theorem 3.2 and Theorem 3.3.]{AlpZin2}, we see that 
\begin{align}\label{l1}
S(\mu)=S(\nu)
\end{align}
and
\begin{align}\label{l2}
W_{2,n}(\mu)=W_{2,n}(\nu). 
\end{align}
Hence \eqref{z1} follows from Theorem \ref{thm4}.

Let $\alpha,\beta=1$. Then 	$K=\{z\in \bb C: (1-z^2)S_n^2(z)\in [0,1]\}$ if and only if  \eqref{z2} is satisfied. Thus by Theorem \ref{thm5}(i), \eqref{l1}, \eqref{l2}, equality   is attained in \eqref{z1} if and only if \eqref{z2} is satisfied. If \eqref{z2} holds then  \eqref{z5} holds in view of \cite[Theorem 3]{PehSte00} and Theorem \ref{thm5} (i). The proof for the cases $\alpha=0, \beta=1$ and $\alpha=1,\beta=0$ is similar.
\end{proof}



\end{document}